\documentclass[11pt]{amsart}

\pdfoutput=1

\usepackage{amsmath}
\usepackage{fullpage}
\usepackage{xspace}
\usepackage[psamsfonts]{amssymb}
\usepackage[latin1]{inputenc}
\usepackage{graphicx,color}
\usepackage{hyperref}
\usepackage{graphicx}

\usepackage{amsmath}%
\usepackage{amsthm}%
\usepackage{amscd}
\usepackage{amsfonts}%
\usepackage{amssymb}%
\usepackage{graphicx}

\usepackage{mathrsfs}

\usepackage{tikz}
\usetikzlibrary{matrix,arrows}

\usepackage{tikz-cd}

\newtheorem{theorem}{Theorem}[section]

\newtheorem{conjecture}[theorem]{Conjecture}

\newtheorem{lemma}[theorem]{Lemma}

\theoremstyle{remark}

\numberwithin{equation}{section}

\newcommand{\Z}{\mathbb{Z}}

\newcommand{\Q}{\mathbb{Q}}

\newcommand{\rad}{\mathrm{rad}}

\makeatletter
\@namedef{subjclassname@2020}{%
  \textup{2020} Mathematics Subject Classification}
\makeatother

  \DeclareFontFamily{U}{wncy}{}
    \DeclareFontShape{U}{wncy}{m}{n}{<->wncyr10}{}
    \DeclareSymbolFont{mcy}{U}{wncy}{m}{n}
    \DeclareMathSymbol{\Sha}{\mathord}{mcy}{"58}

\begin{document}
\title[]{On the $abc$ and the $abcd$ conjectures}

\author{Hector Pasten}
\address{ Departamento de Matem\'aticas,
Pontificia Universidad Cat\'olica de Chile.
Facultad de Matem\'aticas,
4860 Av.\ Vicu\~na Mackenna,
Macul, RM, Chile}
\email[H. Pasten]{hector.pasten@uc.cl}%

\author{Roc\'io Sep\'ulveda-Manzo}
\address{ Departamento de Matem\'aticas,
Pontificia Universidad Cat\'olica de Chile.
Facultad de Matem\'aticas,
4860 Av.\ Vicu\~na Mackenna,
Macul, RM, Chile}
\email[R. Sep\'ulveda-Manzo]{rseplveda@uc.cl}%

\thanks{H.P. was supported by ANID Fondecyt Regular grant 1230507 from Chile.}
\date{\today}
\subjclass[2020]{Primary:  11D75; Secondary:  11J25, 11J97, 11J86} %
\keywords{$abc$ conjecture, linear forms in logarithms, $n$-terms $abc$ conjecture}%

\begin{abstract} We revisit a subexponential bound for the $abc$ conjecture due to the first author, and we establish a variation of it using linear forms in logarithms. As an application, we prove an unconditional subexponential bound towards the $4$-terms $abc$ conjecture under a suitable hypothesis on the size of the variables.
\end{abstract}

\maketitle



\section{Introduction} 

\subsection{The $abc$ conjecture} For a non-zero integer $n$ we let $\rad(n)$ be the largest positive squarefree divisor of it. Let us recall the celebrated $abc$ conjecture of Masser and Oesterl\'e:

\begin{conjecture}[The $abc$ conjecture] Let $\epsilon>0$. There is a number $K_\epsilon>0$ depending only on $\epsilon$ such that the following holds: Given $a,b,c$ coprime positive integers with $a+b=c$, we have 
$$
c\le K_\epsilon \cdot \rad(abc)^{1+\epsilon}.
$$
\end{conjecture}
Without further restrictions, all available unconditional results towards this conjecture \cite{ST, SY1, SY2, MP} take the form
$$
\log c \le K_\epsilon \cdot \rad(abc)^{\alpha+\epsilon}
$$
for some fixed $\alpha>0$. Thus, they are exponential bounds. The sharpest result was obtained in \cite{SY2} with the exponent $\alpha=1/3$.

However, if some restrictions are imposed then subexponential unconditional bounds are available, see \cite{PastenPn21} and the references therein, specially \cite{SY2, PastenTruncated}. Let us recall the following one due to the first author:

\begin{theorem}[Theorem 1.4(1) in \cite{PastenPn21}]\label{Thm0} There is a constant $\kappa>0$ such that the following holds. Let $a,b,c$ be coprime positive integers with $a+b=c$ and suppose that for some $\eta>0$ we have $a\le c^{1-\eta}$. Then
\begin{equation}\label{EqnThm0}
\log c\le \eta^{-1}\exp\left(\kappa\sqrt{(\log \rad(abc))\log_2^* \rad(abc)}\right).
\end{equation}
\end{theorem}

Here, $\log_k$ is the $k$-th iterated logarithm and $\log_k^*(t)=\log_k(t)$ unless it takes a value less than $1$ or it is undefined, in which case we set $\log_k^*(t)=1$. This is a substantial improvement of an earlier bound due to the first author \cite{PastenTruncated} where \eqref{EqnThm0} is replaced by 
\begin{equation}\label{EqnThm0earlier}
\log c\le\eta^{-1}\kappa_\epsilon \exp\left((1+\epsilon)\frac{\log_3^* \rad(abc)}{\log_2^* \rad(abc)}\cdot \log \rad(abc)\right)
\end{equation}
for any $\epsilon>0$, where $\kappa_\epsilon>0$ only depends on $\epsilon$.


\subsection{A variation of the subexponential bound} Our first result is the following variation of the bound \eqref{EqnThm0earlier}. 
\begin{theorem}[Main theorem for $abc$]\label{Thm1} There is a constant $\kappa>0$ such that the following holds. Let $a,b,c$ be coprime positive integers with $a+b=c$ and suppose that for some $\tau>0$ we have 
\begin{equation}\label{EqnCondition}
a\le \frac{c}{\exp\left((\log c)^\tau\log_2^* c\right)}. 
\end{equation}
Then
\begin{equation}\label{EqnThm1}
\log c\le \exp\left(\tau^{-1}\kappa\cdot \frac{\log_3^* \rad(bc)}{\log_2^* \rad(bc)}\cdot \log \rad(bc)\right).
\end{equation}
In particular, if $\tau>0$ is fixed, then we have $\log c\ll_\epsilon \rad(bc)^\epsilon$ for every $\epsilon>0$.
\end{theorem}
One can be more precise about the constant $\kappa$, but this is not relevant for our discussion. Theorem \ref{Thm1} can be deduced from the results in \cite{PastenTruncated} but we prefer to give a self-contained (and somewhat simpler) proof here for the sake of completeness. See Section  \ref{SecProofThm1}.

There are two main differences between \eqref{EqnThm0earlier} and Theorem \ref{Thm1}. First, the condition \eqref{EqnCondition} is less restrictive than the condition $a\le c^{1-\eta}$ when $\eta>0$ is fixed. The second difference is more substantial: In \eqref{EqnThm1} the term $\rad(a)$ does not appear, unlike the bound \eqref{EqnThm0earlier}. This difference turns out to be a key aspect in our application to the $4$-terms $abc$ conjecture to be discussed below.

One can ask whether the bound \eqref{EqnThm1} can be improved to something similar to \eqref{EqnThm0}. At present this seems difficult: The bounds coming from \cite{PastenShimura} are crucial in the proof of Theorem \ref{Thm0}, and these in fact depend on $\rad(abc)$ rather than just $\rad(bc)$.


\subsection{The $n$-terms $abc$ conjecture}

In \cite{BB} Browkin and Brzezi\'nski proposed the following extension of the $abc$ conjecture to the case of $n$ terms, with $n\ge 3$.
\begin{conjecture}[The $n$-conjecture] \label{Conjnterms} Let $n\ge3$. There is a number $M_n$ depending only on $n$ such that the following holds: Given integers $x_1,...,x_n$ satisfying
\begin{itemize}
\item[(i)] $\gcd(x_1,...,x_n)=1$;
\item[(ii)] $x_1+...+x_n=0$; and
\item[(iii)] no proper sub-sum in (ii) vanishes
\end{itemize}
one has that $\max_{1\le j\le n} |x_j|\le \rad(x_1x_2\cdots x_n)^{M_n}$.
\end{conjecture}
In fact, in \cite{BB} it is also conjectured that for every $\epsilon>0$ one can take $M_n=2n-5+\epsilon$ up to finitely many exceptions. See also \cite{Bruin,VojtaABC}.

The previous conjecture is also discussed in \cite{Browkin} and a modification of it is  proposed:

\begin{conjecture}[The strong $n$-conjecture] \label{Conjsnterms} Let $n\ge3$. There is a number $M_n$ depending only on $n$ such that the following holds: Given non-zero integers $x_1,...,x_n$ satisfying
\begin{itemize}
\item[(i)] $\gcd(x_i,x_j)=1$ for all $i\ne j$; and
\item[(ii)] $x_1+...+x_n=0$
\end{itemize}
one has that $\max_{1\le j\le n} |x_j|\le \rad(x_1x_2\cdots x_n)^{M_n}$.
\end{conjecture}


\subsection{The case of four variables: an unconditional result}

While the polynomial analogues of Conjectures \ref{Conjnterms} and \ref{Conjsnterms} are known (see \cite{BB,BM,SS} and the references therein) we are not aware of any unconditional result (not even exponential!) for these conjectures over $\Z$ for any $n\ge 4$. As an application of Theorem \ref{Thm1}, we prove the following unconditional bound for the (strong) $n$-conjecture when $n=4$, usually called the $abcd$ conjecture.

\begin{theorem}[Subexponential bound for the $abcd$ conjecture]\label{ThmABCD} There is an absolute constant $\kappa>0$ such that the following holds: Let $x_1,x_2,x_3,x_4$ be pairwise coprime non-zero integers with 
$$
x_1+x_2+x_3+x_4=0.
$$
Let $H=\max_j|x_j|$ and let us assume that for some $\tau>0$ we have
$$
\min_{i<j}|x_i+x_j|\le \frac{H}{\exp\left((\log H)^\tau\log_2^* H\right)}.
$$
Then, writing $R=\rad(x_1x_2x_3x_4)$, we have
$$
\log H \le \exp\left(\tau^{-1}\kappa\cdot \frac{\log_3^* R}{\log_2^*R}\cdot \log R \right).
$$
In particular, if $\tau>0$ is fixed, then for every $\epsilon>0$ we have $\log H \ll_\epsilon \rad(x_1x_2x_3x_4)^\epsilon$.
\end{theorem}


\section{Subexponential $abc$ without $\rad(a)$} \label{SecProofThm1}

For a rational number $q=u/v$ with $u,v$ coprime integers, its logarithmic height is defined as 
$$
h(q)=\log \max\{|u|,|v|\}.
$$
The following result is essentially due to Matveev \cite{Matveev} and it comes from the theory of linear forms in logarithms. See Theorem 4.2.1 in \cite{EG} for this version.
\begin{lemma}[Linear forms in logarithms]\label{ThmLFL} There is an absolute constant $K>0$ with the following property: Let $\xi_1,...,\xi_m\in \Q^\times$ and let $\xi\ne 1$ be an element in the multiplicative group generated by the numbers $\xi_j$. Then
$$
-\log |1-\xi| \le K^m\cdot \left(\log^* h(\xi)\right)\prod_{j=1}^mh(\xi_j).
$$
\end{lemma}

With this we can prove the following preliminary result:

\begin{theorem}[Preliminary subexponential bound for $abc$]\label{Thm1preliminary} There is a constant $\kappa>0$ such that the following holds: Let $a,b,c$ be coprime positive integers with $a+b=c$. Then
\begin{equation}\label{EqnThm1v2}
\frac{\log (c/a)}{\log_2^* c}\le \exp\left(\kappa\cdot \frac{\log_3^* \rad(bc)}{\log_2^* \rad(bc)}\cdot \log \rad(bc)\right).
\end{equation}
\end{theorem}
\begin{proof} Let $\xi=b/c$ and choose $\xi_j=p_j$ for $j=1,...,m$ as the different prime divisors of $bc$. Then $1-\xi=a/c$ and $h(\xi)=\log c$, and from Lemma \ref{ThmLFL} we obtain
$$
\frac{\log(c/a)}{\log_2^* c}\le K^m\prod_{j=1}^m\log p_j\le \left(\frac{K\log R}{m}\right)^m
$$
where $R=\rad(bc)$ and we used the arithmetic-geometric mean inequality. 

Recall that $m$ is the number of different prime factors of $bc$, hence, of $R$. Thus, from well-known elementary bounds we have
$$
m\le M (\log R)/\log_2^*R
$$
for a suitable constant $M>1$. The function
$$
t\mapsto \left(K(\log R)/t\right)^t\quad\mbox{for $t>0$}
$$ 
is increasing in the range $0<t \le K(\log R)/e$. In particular, adjusting $K$ if necessary to achieve 
$$
M (\log R)/\log_2^*R\le K(\log R)/e, 
$$
 we obtain
$$
 \left(\frac{K\log R}{m}\right)^m \le \left((K/M)\log_2^* R\right)^{M(\log R)/\log_2^* R}
$$
and the result follows.
\end{proof}

\begin{proof}[Proof of Theorem \ref{Thm1}] The assumption \eqref{EqnCondition} gives $(\log (c/a))/\log_2^* c\ge (\log c)^\tau$ and the result follows from Theorem \ref{Thm1preliminary}.
\end{proof}


\section{Application: a bound for the $4$-terms $abc$ conjecture}

\begin{proof}[Proof of Theorem \ref{ThmABCD}] Without loss of generality we may assume that
$$
\min_{i<j} |x_i+x_j| = |x_1+x_2|.
$$
Up to symmetry we have two cases: $H=|x_1|$ or $H=|x_3|$.

In the first case we define $y=x_1+x_2$ and note that the three integers in this equation are coprime and non-zero. Rearranging terms to obtain an equation $a+b=c$ with positive integers we see that $c=|x_1|$ (by maximality of $H=|x_1|$), one can choose $a=|y|$, and assumption \eqref{EqnCondition} is satisfied with these choices. Theorem \ref{Thm1} then gives
$$
\log H = \log |x_1| \le \exp\left(\tau^{-1}\kappa \cdot \frac{\log_3^*\rad(x_1x_2)}{\log_2^*\rad(x_1x_2)}\cdot \log\rad (x_1x_2)\right).
$$
The case $H=|x_3|$ is analyzed in the same way after noticing that $|x_1+x_2|=|x_3+x_4|$.
\end{proof}

\section{Acknowledgments}

H.P. was supported by ANID Fondecyt Regular grant 1230507 from Chile. 


\end{document}